\documentclass{amsart}
\usepackage{ifthen}
\usepackage{enumerate}
\numberwithin{equation}{section}
\def\epsilon{\varepsilon}

\def\de{\delta}
\def\si{\sigma}
\def\ga{\gamma}
\def\lap{\Delta}
\def\ti{\tilde}

\def\ga{\gamma}

\def\beq{\begin{eqnarray}}
  \def\eeq{\end{eqnarray}}
\def\be{\beta}
\def\al{\alpha}
\def\ep{\epsilon}

\def\partt{\frac{\partial }{\partial t} }
\def\parts{\frac{\partial }{\partial s } }
\def\phi{\varphi}
\def\R{\mathbb R}
\def\boundary{\partial}
\def\N{\mathbb N}

\def\part{\partial}
\def\Id{ \ \rm{  Id } }

\def\curlR{\mathcal R}
\def\Sc{{\rm R}} 
\def\grad{\nabla}

\DeclareMathOperator{\vol}{vol}

\DeclareMathOperator{\Riem}{Riem}

\DeclareMathOperator{\Ricci}{Ric}

\DeclareMathOperator{\dist}{dist}

\def\counterword#1{%
  \ifthenelse{\ref{#1}=1}{one}{}%
  \ifthenelse{\ref{#1}=2}{two}{}%
  \ifthenelse{\ref{#1}=3}{three}{}%
  \ifthenelse{\ref{#1}=4}{four}{}%
  \ifthenelse{\ref{#1}=5}{five}{}%
  \ifthenelse{\ref{#1}=6}{sic}{}%
}

\newtheorem{theorem}{Theorem}[section]
\newtheorem{lemma}[theorem]{Lemma}

\newtheorem{proposition}[theorem]{Proposition}
\theoremstyle{definition}

\theoremstyle{remark}
\newtheorem{remark}[theorem]{Remark}

\begin{document}
\title[Local smoothing for the Ricci flow in dimensions two and
three]{Local smoothing results for the Ricci flow in dimensions two and three }

\thanks{We thank Robert Haslhofer for comments on and discussions about an
  earlier version of this paper}

%

\author{Miles Simon}
\address{Miles Simon:
Institut f\"ur Analysis und Numerik (IAN), Universit\"at Magdeburg, Universit\"atsplatz 2, 39106 Magdeburg, Germany}

\curraddr{}
\email{ msimon at ovgu point de}


\dedicatory{}


\begin{abstract}
We present local estimates for solutions to the Ricci flow, without the assumption that the solution
has bounded curvature. These estimates lead to a  generalisation of 
one of the pseudolocality results of G.Perelman in dimension two.
\end{abstract}

\maketitle
\section{Introduction}
In this paper, unless otherwise specified, a solution $(M,g(t))_{t \in
[0,T) }$ to Ricci flow  refers to a family $(M,g(t))_{t \in
[0,T) }$ of smooth ( in space and time ) Riemannian manifolds which are
complete for all $t \in [0,T)$, solve $\partt g(t) = -2\Ricci(g(t))$
and have no boundary.
We do not require (unless otherwise stated) that the
solution has bounded curvature.

In the paper \cite{Pe1}, G.Perelman proved the following fact: if a ball
${{}^0 B}_r(x_0) $  in $(M,g(0)$ at time zero is {\it almost
  euclidean}, and $(M,g(t))_{t \in
[0,T) }$ is a solution to the Ricci flow with bounded curvature, then 
for small times $  t \in [0,\ep(n,r))$, we have estimates on how the
curvature behaves on balls  $ {}^t B_{\ep r}(x_0)$. There are a number of versions of his
theorem: see Theorem 10.1 and Theorem 10.3 in \cite{Pe1} for proofs and the
definitions of {\it almost euclidean}. See \cite{ChTaYu} ,\cite{Lu},\cite{CCGGIIKLLN},\cite{KL}
and \cite{BLC} for alternative
proofs and related results. 
In dimension two, we show that a similar result holds under weaker
initial assumptions.
\begin{theorem}\label{twoD}
Let ${1}>\si,\al>0, v_0, r >0 , N>1$ be given.
Let $(M^2,(g(t)) _{t \in [0,T)}$ be a smooth complete solution to Ricci flow,
$x_0 \in M$, and assume that
\begin{itemize}
\item $\vol{{}^{0} B_{r}(x_0)} \geq v_0r^2$ and
\item $\Sc(g(0)) \geq -  \frac {N}{r^2}$ on  ${}^{0} B_{r}(x_0)$.
\end{itemize} 
Then  there exists a  a $\ti v_0 = \ti v_0 (v_0,\si,N,\al) >0$ and a $\de_0 = \de_0(v_0,\si,N,\al) >0$ such that
\begin{itemize}
\item $\vol({}^{t} B_{r(1-\si)}(x_0)) \geq \ti v_0 r^2 $
\item $\Sc(g(t)) \geq -  \frac {(N + \al)}{r^2} $ on  $ {}^{t} B_{r(1-\si)}(x_0)$
\item $|\Sc(g(t))| \leq \frac{1}{ \de^2_0 t} $  on $ {}^{t} B_{r(1-\si)}(x_0)$
\end{itemize}
as long as $t \leq (\de_0)^2r^2$ and $t \in [0,T)$.
\end{theorem}

 \begin{remark} Notice that we do not require that a region be {\bf
   almost euclidean} here ( see Thm. 10.1 and 10.3 of \cite{Pe1} for the definition
 of  {\it  almost euclidean} ). If the ball   ${}^{0} B_{1}(x_0)$ is {\it
   almost cone like}  ( that is, it  is as close as we like in the
 Gromov Hausdorff sense to an
 euclidean cone and has $\Sc \geq
 -2$ ) then the Theorem (with $r=1$ ) still applies. This means that the interior of regions which are
 cone like in this sense will be smoothed out by
 Ricci flow in two dimensions, {\it regardless}  of what the solution looks like outside
 of this region.
Both of the Theorems of G.Perelman (Thm. 10.1, Thm.10.3 of \cite{Pe1}) do not apply to this situation.
\end{remark}

\begin{remark} By scaling, it suffices to prove the Theorem for
$r=1$. 
\end{remark}
\begin{remark}
 It is not possible to improve the constant $\de_0$ in the estimate
 \break\hfill $|\Riem(g(t))| \leq  \frac{1}{\de_0 t}$
to an arbitrary constant $\de_0>0$ for a short time. This is because,
solutions coming out of non-negatively curved cones exist which have curvature behaviour immediately like $\frac c t$ where $c>0$ depends on the cone angle
( see \cite{SchSi} ).
In G.Perelman's first Pseudolocality result (Theorem 10.1 of
\cite{Pe1}),  where he assumes that a ball
$B_{r}(y_0)$ at time zero is almost euclidean, he showed that it is possible to obtain an estimate of the form
$|\Riem(g(t))| \leq \frac{\al} {t}$ on a smaller ball for arbitrary $\al$ at least for some short time interval 
depending on $\al$,  as long as the initial ball is {\it close enough}
to the euclidean ball. Here {\it close enough} means, that $(\vol
(\boundary \Omega ))^n \geq (1- \de) c_n (\vol (\Omega))^{n-1}$ for
any $\Omega \subset B_{r}(y_0)$ where $c_n$ is the euclidean
isoperimetric constant, $\Sc \geq - \frac 1 {r^2}$ and $\de= \de(n,\al) >0$ is
small enough.
\end{remark}

The second theorem is valid in three dimensions. In contrast to the
above theorem, we need to have information on how the curvature is
behaving (in time) in the balls we are considering in order to draw (stronger) conclusions.

\begin{theorem}\label{threed}
Let $ r, v_0>0 , N >1, {1}> \si>0,V>0 $  be given.
Let $(M^3,(g(t)) _{t \in [0,T)}$ be a smooth complete solution to Ricci flow  with $T\leq 1$ and let $x_0 \in M$ be a point
such that
\begin{itemize}
\item $\vol{{}^{0} B_{r}(x_0)} \geq v_0r^3$ and
\item $\curlR(g(0)) \geq -  \frac {V} { r^2}$ on  ${}^{0} B_{r}(x_0)$,
\item $ |\Riem(g(t))| \leq \frac N { t } $ on ${}^{t} B_{r}(x_0)$, for all $t \in (0,T)$.
\end{itemize} 
Then there exists a $\ti v_0 = \ti v_0 (v_0,N,\si,V) >0$ and a $\de_0 = \de_0(v_0,N,\si,V) >0$ such that
\begin{itemize}
\item $\vol({}^{t} B_{r}(x_0)) \geq  \ti v_0 r^3$
\item $\curlR(g(t)) \geq -  \frac {400 N V} {r^2} $ on  $ {}^{t} B_{r(1-\si)}(x_0).$
\end{itemize}
as long as $t \leq r^2(\de_0)^2$  and $t \in [0,T)$.
\end{theorem}

\begin{remark} By scaling arguments it suffices to prove the theorem
  for $r=1$ and $V
= \frac{1}{400N}$: see Remark \ref{endremark}.
\end{remark}
\begin{remark} As in the two dimensional case (Theorem \ref{twoD} above), the
regions which are considered are not
necessarily {\it almost euclidean} at time zero.
\end{remark}
\begin{remark} Related  results were proved recently in a pre-print
  \cite{ChGuZh}. 
There the authors require that the curvature of the solution be uniformly bounded by
a constant $c$ on a ball for the times $t \in [0,S)$ being
considered. See \cite{ChGuZh}. 
\end{remark}

\begin{remark} The above  results were first presented in Nov. 2011 at
  the H.I.M workshop in Bonn
'Geometric Flows'.
\end{remark}
\section{A local bound for the curvature on regions whose curvature is
  bounded from below}

We use the following notation in this paper.
\noindent {\bf Notation}
\begin{itemize}
\item[] $d(x,y,t) = d_t(x,y) = d(g(t))(x,y)$ is the distance from $x$ to $y$ in $M$ with respect to the metric $g(t)$ 
\item [] $d_t(x) =  d_t(x,x_0)$ is distance from $x$ to $x_0$
  with respect to the metric $g(t)$ for some fixed $x_0$.
 \item[] ${}^t B_r(x) := $ ball of radius $r>0$, centre point $x \in M$ measured with respect to $(M,g(t))$.
That is ${}^t B_r(x) := \{ y \in M | d_t(x,y) <r\}$
\item[] $\vol({}^t B_r(x)):=$ volume of ${}^t B_r(x)$ with respect to the volume form $d\mu_t$ induced by 
$g(t)$
\item[] $\Riem(g(t))(x) = \Riem(x,t)$ is the Riemannian curvature Tensor of the metric $g(t)$ at the point 
$x \in M$.
\item[] $\curlR(g(t))(x):= \curlR(x,t)$ is the curvature operator of $(M,g(t))$ at the point $x \in M$:
$\curlR(x,t)(V,W) := \Riem_{ijkl}(x,t)V^{ij}W^{kl}$ for 2-forms $V,W$
($\Riem_{ijkl}$ is the curvature tensor of $g(t)$ in local
coordinates, and $V=V_{ij}dx^i \otimes dx^j,$ $W = W_{ij}dx^i \otimes
dx^j$, and $V^{km} = V_{ij}g^{ki}g^{mj}$, $W^{km} = W_{ij}g^{ki}g^{mj}$ )
\item[] $\Sc(x,t)$ is the scalar curvature of $(M,g(t))$ at the point
  $x \in M$
\end{itemize}

Let $(M,g(t))_{t \in [0,T]},$ $T \leq 1$ be a smooth complete solution to Ricci flow.
We wish to prove estimates on a ball of radius $r$ at time $t \in
[0,T] $, assuming the curvature operator
stays bounded from below on ${{}^{t}B}_{r}(x)$ for all $t \in [0,T]$ and the volume of 
${{}^{t}B}_{r}(x)$ is bounded from below for all $t  \in [0,T]$.
The estimates will depend on $n$, $r$ and the bounds from below.
A local result of this type was
obtained by B.-L Chen in the proof of Theorem 3.6 of \cite{BLC},  under the assumption that the curvature
operator is non-negative on all of $(M,g(t))_{t \in [0,T)}$.
A global result of this type was obtained in Lemma 2.4 in \cite{Si1} , and
Lemma 4.3 \cite{Si3}

In the proof of  Theorem 3.6 of \cite{BLC} by B.L-Chen,
the author uses a point picking argument of G.Perelman before rescaling to obtain a contradiction to Proposition
11.4 of \cite{Pe1}  ( in the proof Lemma 2.4 in \cite{Si1} , and
Lemma 4.3 in \cite{Si3}  we used a more global point
picking type argument of R.Hamilton and then also obtained  a contradiction
to Proposition 11.4 of \cite{Pe1} after scaling).
The point picking argument of G.Perelman is more suited to this local situation,
and so we use it in the following.

The proof follows the lines given in the proof of  Theorem 3.6 of \cite{BLC} . A number of modifications are
necessary.

\begin{theorem}
  \label{localest}
Let $  r,v_0>0,1>\si >0 $ and $(M^n,g(t))_{t \in [0, T)}$  be a smooth
complete solution to Ricci-flow 
which satisfies
\begin{itemize}
\item[{\rm (a)}] $\vol({{}^{t}B}_{r}(x_0))\geq v_0 r^n $, 
\item[{\rm (b)}] $ \curlR (x,t) \geq -\frac{1}{r^2}$ for all $t  \in [0,T) , x \in  {}^t B_{r}(x_0) $.
\end{itemize}
Then, there exists a $ N = N(n,v_0,\si)< \infty$ such that
\begin{itemize}
 \item[{\rm (c)}]
$|\Riem| \leq \frac{N^2}{ t} + \frac{N^2}{ r^2}$ for all  
$x \in  {}^t B_{r (1-\si)}(x_0), t  \in [0,\frac {r^2}{N^2}) \cap [0,T)$. 
\end{itemize}
\end{theorem}

\begin{proof}
By scaling, it suffices to prove the case $r=1$. 
Assume that the statement is false.
Then we can find solutions $(M_i^n,g_i(t))_{t \in [0, T_i)}$, $T_i
\leq 1$, $i\in\N $, ( $i \neq 0$ for notational reasons: we shall use
the symbol $x_0$ in a moment) and points $x_i \in M_i$ such that
\begin{itemize}
\item[{[a]}]  $\vol(B_{1}(x_i,t)) \geq v_0$ for all $t \in [0,T_i),$  
\item[{[b]}] $ \curlR (x,t) \geq -1$ for all $t  \in [0,T_i) , x \in  {}^t B_{1}(x_i) $ 
\end{itemize}
and points $t_i \in [0,T_i), z_i \in  {{}^{t_i} B}_{(1-\si)}(x_i)$ such that
$t_i \leq \frac{1}{N_i^2}$ and 
$$ |\Riem(z_i,t_i)| \geq \frac{N_i^2}{t_i} + N_i^2$$ with $N_i \to \infty$ as  $i \to \infty$.
Fix $i \in \N$ for the moment and define $M:= M_i$ , $x_0:= x_i$, $s_0 := t_i$, $g(t):=g_i(t)$, $y_0:= z_i$,  $T = T_i$,
$d_t(x) :=  \dist(g_i(t))(x,x_i) = \dist(g(t))(x,x_0) $,
$A = \frac{\si N_i}{8}$, $\ep := \frac{1}{N_i}$ and $\al = N_i^2$.
Then $A \ep \leq \frac {\si}{8}$ and $s_0 \leq \ep^2$
and $g$ solves RF on $[0, T]$ with $T \leq \ep^2$
and $|\Riem|(y_0,s_0) \geq \frac{\al}{s_0} + \frac{1}{\ep^2}$.
That is , we are in the setup of {\bf Claim 1} of Theorem 10.1 of \cite{Pe1}, except he requires $g$ be a RF on $[0,\ep^2]$.
Examining the argument of Perelman, we see that
we only need that $g$ solves RF on $[0,T]$ with $T \leq \ep^2$, as the subsequent point picking argument
only looks at times less than or equal to $s_0$ (see the proof of
Claim 1 of Theorem 10.1 of \cite{Pe1}). Also, we do not have
$d_{s_0}(y_0) \leq \ep$: we have $d_{s_0}(y_0) = d_{s_0}(x_i,z_i) \leq
(1-\si)$. This causes no problem in the point picking argument, and
merely leads to the term $2A\ep +1 -\si$ appearing in place of $2A\ep
+\ep$ in the estimate \eqref{tiddle} below.
Using Claim 1  of Theorem 10.1 of \cite{Pe1}, we obtain new points $\bar y_0 \in M$,  $\bar s_0$ satisfying
\begin{eqnarray}
&& \bar s_0 \leq s_0 \cr
&& d_{ \bar s_0} (\bar y_0 ) \leq 2A \ep +  (1 -\si)\ \ \ ( \leq
3)  \label{tiddle}    \\
&& |\Riem(\bar y_0, \bar s_0)| \geq \frac {\al}{ \bar s_0} =\frac{
  N^2_i}{\bar s_0} \ \ (  \geq N^2_i  ) \label{tiddle2}
\end{eqnarray}
and
\begin{eqnarray*}
|\Riem(x,t)| \leq 4|\Riem(\bar y_0,\bar s_0)| 
\end{eqnarray*} 
whenever
$|\Riem(x,t)| \geq \frac {\al}{t} $, $ t \leq \bar s_0 (\leq s_0) $ and
$d_t(x) \leq d_{  \bar s_0   }(\bar y_0) +A |\Riem|^{-\frac{1}{2}}(\bar y_0,\bar s_0)$.

 Hence a version of Claim 2 of Theorem 10.1 of \cite{Pe1} is
 applicable. We follow the first part of the argument of B.Kleiner/J.Lott in
 the proof of Lemma 32.1  of the Arxiv version of their paper  \cite{KL}.
 This gives us
\begin{eqnarray*}
 |\Riem(x,t)| \leq 4|\Riem(\bar y_0, \bar s_0)|
\end{eqnarray*}
whenever 
\begin{eqnarray}\label{curv1}
&& \bar s_0 - \frac 1 2 \al Q^{-1} \leq t \leq \bar s_0 \ \ \ (**) \
{\rm and } \cr
&& d_t(x) \leq d_{\bar s_0} (\bar y_0) + AQ^{-\frac 1 2}  \ \ \ (*), 
\end{eqnarray}
where here $Q:= \Riem(\bar y_0, \bar s_0)$.
Note (*) just says:
$x \in  {{}^t B}_{d_{\bar s_0}(\bar y_0) + AQ^{ -\frac 1 2 }  }(x_0)$.
Now we modify the rest of the argument of B.Kleiner/J.Lott given in
the proof of
Lemma 32.1  of the Arxiv version of their paper  \cite{KL} 
in order to obtain a product region on which the curvature is bounded.
Notice that we do NOT have $\al \leq \frac 1 {100n}$, and therefore
modifications are necessary. 
We claim that 
\begin{eqnarray*} 
&& \ \ \ \ \  \ \ \ \ \ {}^{\bar s_0}B_{d_{\bar s_0}(\bar y_0) + \frac{1}{10}AQ^{ -\frac 1 2 }  }(x_0)
\subseteq {{}^tB}_{ d_{\bar s_0}(\bar y_0) + \frac{1}{2}AQ^{ -\frac 1
    2}  }(x_0),\cr 
&& {\rm whenever } \ \ \bar s_0 - M_0 Q^{-1} \leq t \leq \bar s_0 
\end{eqnarray*}
where $M_0$ is a fixed large constant. We assume in the following that
$Q$ and $A$ are large (a lot larger than $M_0$).
Let $x \in {}^{\bar s_0}B_{d_{\bar s_0}(\bar y_0) + \frac{1}{10}AQ^{
    -\frac 1 2 }  }(x_0)$. As long as (going backward in time) 
$x \in{{}^tB}_{ d_{\bar s_0}(\bar y_0) + \frac{1}{2}AQ^{ -\frac 1 2}
}(x_0),$ we have $|\Riem(x,t)| \leq 4Q$. Choose $r =
\frac{1}{M^2_0}AQ^{ -\frac 1 2 }$.
Then note that $^t B_r(x) \subseteq{{}^tB}_{ d_{\bar s_0}(\bar y_0) + AQ^{ -\frac 1 2}
}(x_0)$ by the triangle inequality and hence $|\Riem(\cdot,t)| \leq 4Q$ on $^t B_r(x)$ 
( note ${{}^t B}_r(x_0) \subseteq{{}^tB}_{ d_{\bar s_0}(\bar y_0) + AQ^{ -\frac 1 2}
}(x_0)$ trivially, and hence    $|\Riem(\cdot,t)| \leq 4Q$ on ${{}^t
  B}_r(x_0)$ also ) .
Using Lemma 8.3 (b) of \cite{Pe1}
we see that
\begin{eqnarray*}
\partt d_t(x_0,x)(t) \geq -2(n-1)( \frac{2}{3} 4Q \frac{1}{M^2_0}AQ^{
  -\frac 1 2 } + M^2_0 A^{-1}Q^{\frac 1 2}) 
\end{eqnarray*}
Hence
\begin{eqnarray*}
d_t(x_0,x) - d_{\bar s_0}(x_0,x) &&\leq M_0 Q^{-1} 2(n-1)( \frac{2}{3} 4Q \frac{1}{M^2_0}AQ^{
  -\frac 1 2 } + M^2_0 A^{-1}Q^{\frac 1 2})  \cr
&&\leq \frac{8(n-1)}{M_0} AQ^{-\frac 1 2} + 2(n-1) M^3_0 A^{-1}Q^{-\frac 1
  2} \cr 
&&\leq \frac{16(n-1)}{M_0} AQ^{-\frac{1}{2}}\cr
\end{eqnarray*}
where we have used that $A >> M_0$, and $|t-\bar s_0| \leq M_0 Q^{-1}$.
That is 
\begin{eqnarray*}
d_t(x)  && \leq d_{\bar s_0}(x_0,x)  +  \frac{16(n-1)}{M_0}
AQ^{-\frac 1 2} \cr
&& \leq
d_{\bar s_0}(\bar y_0) + \frac{1}{10}AQ^{
    -\frac 1 2 } + \frac {16(n-1)}{M_0} A Q^{-\frac 1 2} \cr
&& \leq d_{\bar s_0}(\bar y_0) + \frac{1}{8}AQ^{
    -\frac 1 2 } 
\end{eqnarray*}
and hence 
$x \in {^{t}B}_{ d_{\bar s_0}(\bar y_0) + \frac{1}{8}AQ^{ -\frac 1 2}
}(x_0)$.
Hence, 
$x \in{{}^tB}_{ d_{\bar s_0}(\bar y_0) + \frac{1}{2}AQ^{ -\frac 1 2}
}(x_0)$ will not be violated for 
$\bar s_0 - M_0 Q^{-1} \leq t \leq \bar s_0$, as claimed.

Now assume $d_{\bar s_0}(x,\bar y_0) \leq \frac{1}{10}AQ^{-\frac{1}{2}
}$ ( i.e. $x \in { {}^{\bar s_0} B}_{\frac{1}{10}AQ^{-\frac 1 2} }(\bar y_0)$ )
and  $\bar s_0 - M_0 Q^{-1} \leq t \leq \bar s_0$.  The triangle inequality implies that
\begin{eqnarray*}
d_{\bar s_0}(x,x_0) && \leq  d_{\bar s_0}(x,\bar y_0) +  d_{\bar
s_0}(x_0,\bar y_0)\cr
&& \leq \frac{1}{10}AQ^{-\frac 1 2} + d_{\bar
s_0}(\bar y_0) 
\end{eqnarray*}
and hence 
\begin{eqnarray*} 
x \in {{}^{\bar s_0} B}_{\frac{1}{10}AQ^{-\frac 1 2}   }(\bar y_0) \subseteq  {{}^{\bar s_0} B}_{ \frac{1}{10}AQ^{-\frac 1 2} + d_{\bar
    s_0}(\bar y_0)  }(x_0)  
\subseteq{ {}^{t} B}_{\frac{1}{2}AQ^{-\frac 1 2} + d_{\bar s_0}(\bar y_0)  }(x_0)
\end{eqnarray*}
( as we just showed ) 
and hence
$|\Riem(x,t) \leq 4Q$ in view of \eqref{curv1}. That is 
$|\Riem(\cdot,t) \leq 4Q$ on $^{\bar s_0}
B_{\frac{1}{10}AQ^{-\frac 1 2}   }(\bar y_0)$ for
all $ \bar s_0 - M_0 Q^{-1} \leq t \leq \bar s_0$. 
Furthermore, for such $x$ and $t$ we have
$x \in {{}^{t} B}_{\frac{1}{2}AQ^{-\frac 1 2} + d_{\bar s_0}(\bar y_0 )}(x_0)
$, as we just showed, and using  \eqref{tiddle}, we see that
\begin{eqnarray*}
\frac{1}{2}AQ^{-\frac 1 2} + d_{\bar s_0}(\bar y_0)
&& \leq \frac{1}{2}AQ^{-\frac 1 2}   + 2A\ep + (1-\si) \cr
&& \leq \frac{\si}{4} + \frac{\si}{4} +  (1-\si) \cr
&& \leq (1- \frac {\si} 2),
\end{eqnarray*}
which gives us that $x \in {{}^{t} B}_{(1- \frac{\si}{2})}(x_0) $, 
and there we have that $\curlR \geq -1$.
Note we have used here that $Q \geq N^2_i$ ( follows from the
inequality \eqref{tiddle2} ) and the
definition of $A$ and $\ep$.
Using the Bishop-Gromov volume comparison principle, 
we also see that
\begin{eqnarray*}
\vol( {{}^{t} B}_{s}(x)) \geq \ti v (\si,v_0)s^n
\end{eqnarray*}
for such $x$ and $t$  and all $s \leq \frac {\si}{10}$  in view of the fact that 
$\vol ({{}^{t}
B}_{1}(x_0)) \geq v_0$.
Taking $x = \bar y_0 \in {{}^{\bar s_0}B}_{\frac{1}{10}AQ^{-\frac 1 2}   }(\bar y_0)$
we get \begin{eqnarray*}
\vol( {{}^{t} B}_{s}(\bar y_0)) \geq \ti v (\si,v_0)s^n
\end{eqnarray*}
for all $s \leq \frac {\si}{10}$ and $\bar s_0 - M_0 Q^{-1} \leq t \leq \bar s_0$.

Defining $\bar z_i := \bar y_0, \bar t_i := \bar s_0$ and substituting
$\al = N_i$ and so on back into the above we get
\begin{eqnarray*}
&& |\Riem(x,t)| \leq 4Q_i \cr
&&\curlR(x,t) \geq -1
\end{eqnarray*}
whenever 
\begin{eqnarray}
&& \bar t_i - M_0 Q^{-1} \leq t \leq \bar t_i \cr
&& d_{\bar t_i}(x,\bar z_i) \leq \frac{1}{10}A_iQ_i^{- \frac 1 2} 
\label{goodx}
\end{eqnarray}
where $Q_i:=|\Riem|(\bar z_i, \bar t_i)| \geq N^2_i$ and $A_i =
\frac{\si N_i}{8}.$
We also have the volume estimate
\begin{eqnarray*}
\vol( {{}^{t} B}_{s}(\bar z_i)) \geq \ti v (\si,v_0)s^n
\end{eqnarray*}
for all $s \leq \frac {\si}{10}$, $ \bar t_i - M_0 Q^{-1} \leq t \leq \bar t_i$.

Rescaling the solutions by $Q_i$ and shifting time by $t_i$ we get solutions to Ricci flow
with
\begin{eqnarray*}
 && |\Riem(x,t)| \leq 4 \cr 
&& \curlR(x,t) \geq -\frac {1}{Q_i}
\end{eqnarray*}
whenever 
\begin{eqnarray*}
&& - M_0  \leq t \leq 0 \cr
&& d_{0}(x,\bar z_i) \leq \frac{\si N_i}{80},
\end{eqnarray*}
in view of the definition of $A_i = \frac{\si N_i}{8}$.
After scaling, the volume estimate is
\begin{eqnarray*}
\vol( {{}^{t} B}_{s}(\bar z_i)) \geq \ti v (\si,v_0)s^n,
\end{eqnarray*}
for all $s \leq \frac { \sqrt{Q_i} \si}{10}$, $- M_0 Q^{-1} \leq t \leq 0$ 
as this is a scale invariant quantity.
$Q_i \to \infty$ as $i \to \infty$ since $Q_i := |\Riem(\bar z_i,\bar
t_i)| \geq N^2_i$ , as one sees from \eqref{tiddle2}. Let us denote
these rescaled solutions also by $(M_i,g_i(t))$. Hence
the bound from below for the curvature operator goes to zero as $i \to \infty$.
Taking the pointed limit of a subsequence as  $i \to
\infty$ ( see Theorem 1.2 of \cite{HaCo} ) of $(M_i,g_i(t),\bar z_i)_{t \in (-  M_0,0]}$, we see that
the limiting solution, denoted by $(\Omega,p_0,h(t))_{t \in (-M_0,0]}$,  has 
non-negative curvature operator, is complete,  has bounded curvature 
$|\Riem(x,t)| \leq 4$ at all times and
points in the limiting manifold, has $|\Riem(p_0,0)| = 1$ and 
 $\lim_ {r \to \infty} \frac{\vol({}^{t} B_r(p_0))} {r^n} \geq \ti v >
 0$ ( note : $\ti v = \ti v(\si,v_0,n)>0$ does NOT depend on $M_0$ ).
We repeat the procedure for larger and larger $M_0$, $M_0 \to \infty$
to obtain, after taking a pointed limit of a subsequence, a solution
 $(\ti  \Omega,\ti p_0,\ti h(t))_{t \in (-\infty,0]}$,  which has 
non-negative curvature operator, is complete,  has bounded curvature 
$|\Riem(x,t)| \leq 4$ at all times and
points in the limiting manifold, has $|\Riem(\ti p_0,0)| = 1$ and 
 $\lim_ {r \to \infty} \frac{\vol({}^{t} B_r(\ti p_0))} {r^n} \geq \ti v >
 0$.
This contradicts Proposition 11.4 in \cite{Pe1}  of G.Perelman.
\end{proof}

Examining the proof, we see that a bound from below on $\curlR$ is
sufficient to obtain an estimate.

\begin{theorem}
  \label{localest2}
Let $ V,r,v_0>0,1>\si >0 $ and $(M^n,g(t))_{t \in [0, T)}$  be a
smooth complete solution to Ricci-flow 
which satisfies
\begin{itemize}
\item[{\rm (a)}] $\vol({{}^{t}B}_{r}(x_0))\geq v_0 r^n $, 
\item[{\rm (b)}] $ \curlR (x,t) \geq -\frac{V}{r^2}$ for all $t  \in [0,T) , x \in  {}^t B_{r}(x_0) $.
\end{itemize}
Then, there exists a $ N = N(n,V,v_0,\si)< \infty$ such that
\begin{itemize}
 \item[{\rm (c)}]
$|\Riem| \leq \frac{N^2}{ t} + \frac{N^2}{ r^2}$ for all  
$x \in  {}^t B_{r (1-\si)}(x_0), t  \in [0,\frac {r^2} {N^2}) \cap [0,T)$. 
\end{itemize}
\end{theorem}

\begin{proof}
In the use of the Bishop-Gromov estimate in the proof above, we obtain a different
constant.
Also, the bound from below on $\curlR $ is now $\curlR \geq -V$.
Otherwise the proof remains unchanged.
\end{proof}


\section{A cut-off function and it's properties}
In the next section we use a cut off function with certain nice properties.
We define this cut-off function here and examine some of it's properties.
\begin{lemma} 
There exists a smooth cut off function $\phi:[0, \infty) \to \R_0^+$
with the following properties.
\begin{itemize} 
 \item[(i)] $0 \leq \phi \leq 1$,
\item[(ii)] $\phi(r) = 1$ for all $r \leq 1$, 
$\phi(r) = 0$  for all $r \geq 2 $,
\item[(iii)] $\phi$ is decreasing: $\phi' \leq 0$,
\item[(iv)] $\phi^{''} \geq -200 \phi$,
\item[(v)] $|\phi^{'}|^2 \leq 200 \phi^{\frac 3 2}$ for some constant $0<C<\infty$.
\end{itemize}

\end{lemma}
\begin{proof}
To construct a cut-off function with the properties (i)-(iv) stated
above is standard. In fact we obtain
 $\phi^{''} \geq -10 \phi$ and  $(\phi')^2 \leq 10 \phi$ in place of (iv).
Define $\psi = \phi^4$.
Then $\psi$ satisfies properties (i)-(iii) trivially, and
$\psi' = ( \phi^4)' = 4 \phi^3 \phi'$ . Then $ (\psi')^2 \leq 16 \phi^6 (\phi')^2 \leq 160 \phi^7 = 160
(\phi^4)^{\frac 7 4} 
= 160 \psi^{\frac 7 4} \leq 160 \psi^{\frac 6 4} = 160 \psi^{\frac 3 2}$
in view of the fact that $\phi \leq 1$. Also
$\psi'' =  ( \phi^4)'' = (4 \phi^3 \phi')' = 12 \phi^2 |\phi'|^2 + 4 \phi^3 \phi''
\geq -40 (\phi^3)\phi = -40 \psi$.
Hence (iv) and (v) are also satisfied.
\end{proof}

 \begin{lemma}\label{cut}
Let $A,B >0$.
We may choose a cut-off function satisfying 
\begin{itemize}
 \item[(i)] $0 \leq \phi \leq 1$
\item[(ii)]  $ \phi(r) = 1$ for all $r \leq A$, 
$\phi(r) = 0$  for all $r \geq A + B$
\item[(iii)] $\phi$ is decreasing: $\phi' \leq 0$.
\item[(iv)] $\phi^{''} \geq -k_0(A,B) \phi$, $(\phi')^2 \leq k_0(A,B) \phi$,
\item[(v)] $|\phi^{'}|^2 \leq k_0(A,B) \phi^{\frac 3 2}$ for some constant $0<k_0(A,B)<\infty$.
\end{itemize}
\end{lemma}
\begin{proof}
By shifting and scaling:
Define $\ti \phi(r) := \phi( \frac{r + B - A}{B}  )$  where $\phi$ is
the function appearing in the above Lemma. Then $\ti \phi$ has all
of the desired properties
\end{proof}

{\bf Construction of a cut-off function on a Riemannian manifold which
  is evolving by Ricci flow.}

\noindent Now we construct a cut-off function similar to that
constructed by G.Perelman (see proof of Theorem 10.1 in \cite{Pe1})
and similar to that used by B.-L. Chen in \cite{BLC}. Assume that we have a solution
to Ricci flow $(M,g(t))_{t \in [0,T)}$.
We do not assume that the curvature is bounded uniformly on some region
for all $t \in [0,T)$ as in the argument of B.L-Chen in the
proof of proposition 2.1 in \cite{BLC}.
Instead we assume a  uniform estimate of the form
\begin{itemize}
\item[(c)] $|\Riem| \leq  \frac{c_0}{t}$ on ${{}^tB}_{\frac 1 4}(x_0)$ for
$t \in [0,S)$
\end{itemize}
for some $S \leq \frac 1 {100}$, $S \leq T$.
Note: The radius of the ball $ \frac 1 4$ is chosen for
convenience. If we replace $ \frac 1 4$  by $\si >0$, then all
constants occurring in this section also depend on $\si$.
This estimate combined with Lemma 8.3  of \cite{Pe1}
guarantees that the cut-off function we construct will satisfy
estimates which are sufficient for the arguments in the following section.

Let $\phi: [0, \infty) \to \R^+_0$ be one of the  cut-off functions
defined above with $A \leq 1$.

Let $r_0(t) = \sqrt{t}$ and $K(t) = \frac{c_0}{t}$.
Then $|\Ricci(x,t)| \leq (n-1)K$ whenever $d(x,x_0,t) \leq r_0(t)$
for all $t \leq S $ in view of (c). Hence, using Lemma 8.3 of \cite{Pe1}, we have
\begin{eqnarray*}
\partt d_t(x) - \lap d_t (x) && \geq -(n-1)(\frac 2 3 Kr_0 + r_0^{-1}) \cr  
&& = -(n-1)\frac{( \frac 2 3 c_0 + 1)}{\sqrt{t}} ,
\end{eqnarray*}
in view of condition (c), 
where this inequality is valid for points $(x,t)$ where $d_t(x)= d(x_0,x,t)$ is differentiable and $t \leq S$,
and $d(x_0,x,t) \geq r_0(t) = \sqrt{t}$. Note that for $t\leq \frac{A^2}{10}$,
the last condition is satisfied 
for all $x$ outside of ${{}^tB}_{\frac A 2}(x_0)$.
The right hand side is integrable.

That is,
\begin{eqnarray*}
&& \partt (d_t(x) + 4(n-1)( c_0 + 1)\sqrt{t})  - \lap (d_t(x) + 4(n-1) (c_0 + 1)\sqrt{t}) \cr
 && \ \ \ \  \ \ \ \geq  \frac{2(n-1)(c_0 + 1)}{\sqrt{t}}  >0.
\end{eqnarray*}
for such points.
Let us denote the constant appearing here as $m_0 = m_0(c_0,n) =
(n-1)(c_0 + 1) $.
Let $k(x,t) = \phi(d_t(x) +  4m_0 \sqrt{t})$.
Using the above information, we obtain the following  evolution inequality for $k$
\begin{eqnarray}
(\partt - \lap) k&& = \phi' \cdot (\partt - \lap ) (d_t + 4m_0 \sqrt{t})(x,t) 
-(\phi'') \cr
&& \leq \phi'  \frac{m_0}{2 \sqrt{t}} +  k_0 \phi \leq k_0 \phi\label{wap}
\end{eqnarray}
where $k_0= k_0(A,B)$ comes from the above Lemma, Lemma  \ref{cut}.
 Note that $\phi( d_t(x) +  4m_0 \sqrt{t})  = 1$ for all 
$x \in {{}^tB}_{\frac A 2}(x_0)$ and $t \leq \frac{A^2}{100m_0^2}$ and hence 
 $\phi' = 0$ for all points $x$ inside ${{}^t}B_{\frac A 2}(x_0)$ 
as long as $t \leq \frac{A^2}{100m_0^2}$ , and hence
the above evolution inequality \eqref{wap} is valid for all $x\in M$ and
all $t \leq \frac{A^2}{100m_0^2}$ ( we assume that $m_0 >>1$ ) as long
 $(x,t)$ is a point where $d_t(x)= d(x_0,x,t)$ is differentiable.
Hence $ h(x,t):= e^{ -2k_0 t} k(x,t)$ satisfies
\begin{eqnarray*}
(\partt - \lap) h(x,t) \leq
  \leq 0
\end{eqnarray*}
for all $x$ and all  $t  \leq \frac{A^2}{100m_0^2}$, as long as
$d(x_0,\cdot,\cdot)$ is differentiable there.
We collect the definitions and observations made above in the following.

\begin{proposition}\label{cutty}
Let $(M,g(t))_{t \in [0,T)}$ by a smooth complete solution to Ricci flow
and $\phi$ be one of the functions appearing in Lemma \ref{cut}.
 with $A \leq 100$.
We assume that 
\begin{itemize}
\item[(c)] $|\Riem| \leq  \frac{c_0}{t}$ on ${{}^tB}_{\frac 1 4}(x_0)$ for
$t \in [0,S)$
\end{itemize}
for some $S \leq \frac 1 {100}$, $S \leq T$.
$h:M \to \R$ is the function $ h(x,t):= e^{-2k_0t} ( \phi(d_t(x) +
4m_0 \sqrt{t}))$ where $d_t(x):= d_t(x_0,x)$, and $x_0$ is a fixed point
in $M$ and $m_0 = m_0(c_0,n) =
(n-1)(c_0 + 1) $, $k_0 = k_0(A,B)$.
For $t \leq \frac{A^2}{100m_0^2}$,
we have
\begin{eqnarray*}
(\partt - \lap) h(x,t) \leq 0,
\end{eqnarray*}
as long as $d(x_0,\cdot,\cdot)$ is differential at $(x,t)$.
$ h \equiv 0 $ for all $d_t(x) \geq (A+ B)$
and $ h \equiv e^{-2k_0t}$ for all $d_t(x) \leq A - 4m_0 \sqrt t $ and
$ h(x,t) \leq e^{-2k_0t}  \leq  1$.

\end{proposition}

\section{A local result in two dimensions}\label{local sec}
In this section we restrict ourselves to the two dimensional case.
We consider  a ball of radius $r$ in a two dimensional manifold which
has curvature operator and volume 
bounded from below by known constants.
We show that a ball of smaller radius will smooth out quickly at least for a short time.
The rate of this smoothing depends on the bounds from below and $r$.


\begin{theorem}\label{2dd}
Let $(M^2, g(t)) _{t \in [0,T)}$ be a smooth complete solution to Ricci flow
and let $x_0 \in M$, $N, v_0,r>0$ and $1>\si,\al >0$.
Assume that
\begin{itemize}
\item $\vol{{}^{0} B_{r}(x_0)} \geq v_0 r^2$ and
\item $\Sc(g(0)) \geq -  \frac {N}{r^2}$ on  ${}^{0} B_{r}(x_0)$,
\end{itemize} 
and $1>\si >0$.
Then there exists a $ \ti v_0 = \ti v_0(v_0,\si,N ,\al)>0 $ and  a $\de_0 = \de_0(v_0,\si,N,\al) >0$ such that
\begin{itemize}
\item $\vol({{}^{t} B}_{r}(x_0)) \geq  \ti v_0 r^2$
\item $\Sc(g(t)) \geq -  \frac{(N + \al)}{r^2}$ on  $ {}^{t} B_{(1-\si)r}(x_0).$
\item $|\Sc(g(t))| \leq \frac{1}{\de_0 t} $  on $ {}^{t} B_{(1-\si)r}(x_0)$
\end{itemize}
as long as $t \leq r^2(\de_0)^2$  and $t \in [0,T)$.
\end{theorem}



\begin{remark}
In Theorem 3.1 of \cite{BLC} B.-L. Chen proved the following similar result: if we
assume the above conditions with $r = 1$  but replace the lower bound on the scalar
curvature by the condition that $    | \Sc (g_0)| \leq 1$ on ${{}^0B}_2(x_0)$, then 
$ |\Sc(g(t)|\leq 2$ for all $t \in [0,T(n,v_0))$ on a smaller ball. This a version 
of G.Perelman's second Pseudolocality result, Theorem 10.3 of \cite{Pe1} in dimension two.
Note that in this case,  the curvature bound and volume bound from below guarantee
that balls of radius $r \leq R = R(n,\ep,v_0)$ which are sufficiently
small satisfy the {\it  almost euclidean} condition $\vol({{}^{0} B}_{r}(x_0)) \geq (1- \ep) r^2$ .
\end{remark}

\begin{proof}  
By scaling, it suffices to prove the theorem for $r=1$. W.l.o.g. $\si
> \frac 1 2$. 
By the Bishop-Gromov volume comparison principle we have 
$\vol(   {{}^0 B}_{s}(x_0)) \geq c(N,v_0)s^n $ for all $s \leq 1$.
In particular 
$\vol( {{}^0 B}_{\frac 1 {1000} } (x_0)) \geq \ti v_0(N,v_0)>0 $.
For some maximal time interval $[0,S_{max})$, $0<S_{max} \leq T$ we have (due to smoothness) that  
\begin{itemize}
 \item[(a)]  $\Sc \geq - (N + 2\al) $ on ${{}^t B}_{(1-\si)}(x_0)$
\item[(b)] $\vol({{}^t B}_{\frac {1}{1000}}(x_0)) \geq \frac {\ti v_0}{2}  $
\end{itemize} 

Note that $(b)$ implies $(\ti b) : \vol( {{}^t B}_{1 } (x_0)) \geq \frac{\ti
v_0(N,v_0)}{2}>0 $ trivially.

Our aim is to obtain an estimate from below for the time $S_{max}$ which
only depends on $N,v_0,\si,\al$ ($n = 2$ is fixed here).
According to Theorem \ref{localest} above, we have that 
\begin{itemize}
\item[(c)] $|\Sc| \leq \frac{c_0(\ti v_0,N,\si,\al)}{t}$ on
  ${{}^tB}_{\frac 1 4}(x_0)$ \hfill\break for
$t \in [0,S_{max})\cap [0,S(N,\ti v_0,\si,\al)) =: [0,S(N,v_0,\si,\al) )$
\end{itemize}
for some constant $c_0 = c_0(N,\ti v_0,\si,\al ) = c_0(N,v_0,\si,\al) $.
In the rest of the proof we often shorten time intervals  $[0,S(N,v_0,\si,\al) )$
to $[0, \ti S(N,v_0,\si,\al) )$ where $0<\ti S(N,v_0,\si,\al) <S(N,v_0,\si,\al)
$. We will denote $\ti S(N,v_0,\si,\al) $ also by $S(N,v_0,\si,\al)$.

{\bf Claim (i)  } The scalar curvature is bounded from below by $-(N +
\al)$ on
${{}^t B}_{(1-\si)}(x_0)$
for $t \leq S$ where $S = S(v_0, \si,\al , N) >0$, as long as $t
\leq S_{max}$.
That is, (a) is not violated for $t \leq  S$ as long as (b) still
holds.

\noindent proof of Claim (i):
\noindent
We modify the argument of B.-L.Chen given in the proof of Theorem 3.6 in \cite{BLC}.
Let $f:= h  \Sc$ where we have chosen $\phi$ in the definition of $h$
of Proposition \ref{cutty}  to be a smooth function
with $\phi(r) = 0$ for $r \geq (1 -\frac \si 4)$, and $\phi(r) = 1$ for
$r\leq (1-  \frac \si 2)$.
Note that $m_0 = m_0(c_0,n) = m_0(c_0(v_0,N,\si,\al),2) = m_0(v_0,\si,\al,N)$ in this
case,
where $m_0$ is the constant appearing in Proposition \ref{cutty} .
In the following we assume without loss of generality, that  $t \leq
\frac{A^2 = (1-\frac{\si}{2})^2 }{100m_0^2}$, so that Proposition \ref{cutty} is valid.

Using the evolution inequality  for $h$ and the evolution equation
for $\Sc$ we see that at any point $(x,t)$ where  $\Sc(x,t) < 0$ and
$d(x_0, \cdot, \cdot) : M \times [0,T) \to \R $ is differentiable we have
\begin{eqnarray}
&& \partt ( h \Sc + \sqrt{t}  )  - \lap( h \Sc + \sqrt{t}  ) \cr
&& \ \ = \Sc (\partt -\lap)( h)  + h (\partt - \lap) ( \Sc)  - 2g( \grad
h, \grad \Sc ) 
+ \frac{1}{2 \sqrt{t}   }\cr
&& \ \ \geq 2h |\Ricci|^2  - 2g( \grad
h, \grad \Sc ) 
+ \frac{1}{2 \sqrt{t}.  }  \label{imp}
\end{eqnarray}
If $(x,t)$ is a first time and point where 
$( h(x,t) \Sc(x,t) +
\sqrt{t}) =-(N + \frac {\al}{2} )$ then 
the gradient term at $(x,t)$ can be estimated as follows.
\begin{eqnarray*}
- 2g(\grad \Sc, \grad h)   && =
  - \frac 2 h g(\grad (\Sc h ), \grad h)
+ 2 \Sc \frac{|\grad h|^2 }{h} \cr
&& =   2 \Sc \frac{|\grad h|^2 }{h} \cr
&& \geq -\frac{(\Sc)^2}{4} h - (\frac 4 {h^3})   |\grad h|^4\cr
&& \geq -\frac{(\Sc)^2}{4} h - 2C(\si) \cr
\end{eqnarray*}
where in the last line we have used that
$|\grad d| = 1$ and 
$|\phi'|^2 \leq C(A,B) \phi^{\frac 3 2}$ with 
 $B = \frac \si 4$, 
$A =  (1-  \frac \si 2)$.
Now using this inequality in \eqref{imp} we get
\begin{eqnarray*}
\partt ( h \Sc +  \sqrt{t}  )  - \lap( h \Sc +  \sqrt{t}  ) &&
\geq 2h |\Ricci|^2 -\frac{(\Sc)^2}{n} h - 2C  +  \frac{1}{2 \sqrt{t}   } \cr
&&> 0
\end{eqnarray*}
at the point $(x,t)$ in question, since $( h(x,t) \Sc(x,t) +
\sqrt{t}) = - (N + \frac{\al}{2})$  guarantees that $\Sc(x,t) < 0$, as long
 as $d(x_0, \cdot, \cdot)  $ is differentiable at $(x,t)$ and $t \leq S(v_0,N,\si,\al)$.
Hence, in view of the maximum principle, we see that $ h \Sc +
\sqrt{t} \geq 0$ for all 
$t \leq  S(N,\si,v_0,\al)$ as long as $t \leq
S_{max}$
( for the case that $(x,t)$ is not a point where $d$ is differentiable,
then the argument is still valid, as we explain in Claim (iii) at the end
of the proof).
In particular, this shows that $\Sc \geq -(N +\frac{3\al}{4})$ for $x
\in {{}^tB}_{1-\si}(x_0)$ as long as 
$t \leq  S(N,\si,v_0,\al)$ (possibly a smaller $S$) and $t \leq
S_{max}$, in view of the definition
 $ h(x,t):= e^{-2k_0t} ( \phi(d_t(x) +
8m_0 \sqrt{t})$, 
which is close as we like to one on ${{}^tB}_{1-\si}(x_0)$
for $t \leq  S(N,\si,v_0,\al)$ small enough.

This finishes the proof of the {\bf Claim (i)}.

\noindent {\bf Claim (ii)} 

The volume condition (b) is not violated for a well
defined time interval, as long as (a) holds.
Let $x, y \in {{}^tB}_{\frac {1}{100}}(x_0)$ then $d(x,y,t) \leq d(x,x_0,t) +
d(y,x_0,t) \leq \frac{2}{100}$ and hence any  shortest
geodesic between $x$ and $y$ must lie in ${{}^t B}_{\frac 1 4}(x_0)$
(proof: If it didn't, smooth out the union of the two radial curves (measured with respect to $g(t)$)
going from $x$ to $x_0$ and then back to $y$ of length $\frac{2}{100}$. This would
result in a curve of length less
then $\frac{3}{100}$.
Any curve which starts in  $ {{}^tB}_{\frac{1}{100}}(x_0)$ , reaches
$\boundary {{}^tB}_{\frac 1 4}(x_0)$ and finishes
in  $ {{}^tB}_{\frac 1 {100}}(x_0)$ must  have length larger than or equal to
$\frac {1}{ 20}$.
Hence, if a minimising Geodesic between $x$ and $y$ leaves
${{}^tB}_{\frac 1 4}(x_0)$ we obtain a contradiction).
Hence using the estimate of Hamilton (see Theorem 17.2 of \cite{HaFo} and the Editors' comment thereon
in \cite{CaChChYa} or, alternatively, see Appendix B in \cite{Si3} )
and the fact that (a) and  (c) hold on $ {{}^tB}_{\frac 1 4 }(x_0)$, we get
\begin{eqnarray*}
  && -(N + 2\al)d(y,x,t)   \geq \partt d(y,x,t) \geq - \frac{c_1(c_0)}{\sqrt{t}}  \cr
&& \mbox{ for all } \ \  s \leq t \leq \min(S_{max},S(N,v_0,\si,\al)) \ \ x,y \in
{{}^tB}_{\frac 1 {100}}(x_0),
\end{eqnarray*}
where $r \geq \partt d \geq m$ is meant in the sense of forward
difference quotients ( see Theorem 17.2 of \cite{HaFo} ).
Note $c_1(c_0) = c_1(v_0,\si,N,\al)$.
Integrating in time we get
\begin{eqnarray*}
 && e^{-(N+\al)(t-s)} d(x_0,x,s)  \geq d(x,y,t) \geq d(x,y,s) - 2c_1 \sqrt{ t} \cr
&&\ \ \mbox{ for all } \ \  s \leq t \leq \min(S_{max},S(N,v_0,\si,\al))  \ \ x,y \in
{{}^tB}_{\frac 1 {100}}(x_0)
\end{eqnarray*}
Arguing as in Corollary 6.2 of \cite{Si3}, we see that
$\vol({^{t}B}_{\frac 1 {1000}}(x_0)) \geq \frac 3 4 \ti v_0$ for all
$t \leq S(N,v_0,\si,\al) $ ( we have possibly decreased $S$ once again
here) as long as $t \leq S_{max}$. That is the volume condition is not
violated for the time interval $[0,S(N,\si,v_0,\al)]$  (as long as $t \leq S_{max}$).

In particular we see that the second condition (b) will not be violated for some well defined time interval
$[0,S(N,\si,v_0,\al)]$  (as long as $t \leq S_{max}$).
This finishes the proof of  {\bf Claim (ii)} and of the theorem
if we accept Claim (iii) below.

\noindent{\bf Claim (iii)} 
If $d_t(\cdot)$ is not
differentiable at $x \in M$ then we use the trick of E.Calabi
(\cite{Ca}) as follows.
Let $y_0$ be a point on a shortest geodesic between $x_0$ and $x$
which is very close to $x_0$. By smoothness, we can find a small open 
neighbourhood $P$ of $(x,t)$ in  $M \times (0,T)$ such that
$d(y_0,\cdot,s)$ is differentiable at $y$ for each $(y,s) \in P$.
We define $\ti d_s(y) = d(x_0,y_0,s) + d(y_0,y,s)$. Then   $\ti
d_s(\cdot)$ is differentiable at $y$ for all $(y,s) \in P$.
Furthermore, $\ti d_s(\cdot) \geq d_s(\cdot)$ for all $(y,s) \in P $ due to
the triangle inequality. Since $\phi$ is non-increasing, we therefore
have $\ti k = \phi(\ti d + m_0 \sqrt s) \leq \phi(d + m_0 \sqrt s)  =
k$ in $P$
and per definition  $\ti k(x,t) = k(x,t)$ if $(x,t)$ is the point
given at the beginning of the claim.
Also , if we pick $y_0$ very close to $x_0$ we still have
\begin{eqnarray*}
&& \partt (\ti d_t(x) + m_0 \sqrt{t})  - \lap (\ti d_t(x) + m_0\sqrt{t}) \cr
 && \ \ \ \  \ \ \ \geq  \frac{m_0}{2\sqrt{t}}  >0.
\end{eqnarray*}
where here $m_0$ is the constant appearing in Proposition \ref{cutty}.
Hence we may argue with
$ \ti h(y,s):= e^{-2k_0 s} \ti k(y,s)$ 
everywhere above.
If for example $(x,t)$  is a first time and point where $( h(x,t) \Sc(x,t) +
\sqrt{t}) = -(N +\frac {\al} {2})$ ,  then $(x,t)$ is a first time and point for which 
the function  $( \ti h(x,t) \Sc(x,t) +
\sqrt{t}) =  -(N +\frac {\al} {2})$ on the set $P$  and hence we may argue as above with  $( h(x,t) \Sc(x,t) +
\sqrt{t})$ replaced by $( \ti h(x,t) \Sc(x,t) +
\sqrt{t}) $  ( note that without loss of generality $\Sc <0$ on $P$,
since $ \Sc(x,t) <0$ and hence 
$ \ti h(y,s) \Sc(y,s) +
\sqrt{t}  \geq  h(y,s) \Sc(y,s) +
\sqrt{t} \geq  -(N +\frac {\al} {2})$ on $P \cap \{ (y,s) | s \leq t \}$).
We must also consider the case that $\ti d_{(\cdot)}(x)$ is not
differentiable in time
at the time $t$ we are considering.
In this case, all the estimates are still valid if we understand the
inequalities $ \partt \ti d_t(x) \geq m$
or $\partt \ti  d_t(x) \leq m$ in the sense of forward difference
quotients:  see \cite{HaFour}.
At times $s<t$ very close to $t$ ( $(x,t)$ as above ), we have ( due to smoothness ) 
\begin{eqnarray*}
\lap (\ti h(x,s) \Sc(x,s) +
\sqrt{s})  \geq - e , \\
| \grad (\ti h(x,s) \Sc(x,s) ) |  \leq  e
\end{eqnarray*}
where $e$ is as small as we like.
Remembering that $\ti h(x,t) >0$, we see that the term 
$ - \frac 2 {\ti h(x,s)} g(\grad (\Sc \ti h ), \grad \ti h)(x,s)$ which is zero at $(x,t)$
is also as small as we like for $s<t$ very close to $t$.
Hence, examining the proof of Claim (i)  again, we see that
$$\partt( \ti h(x,s)  \Sc(x,s) +
\sqrt{s})  >0$$  in the sense of forward difference  quotients for $s<t$ close
to $t$.

In particular, using Lemma 3.1 in in \cite{HaFour}, we see that 
$\ti h(x,t) \Sc(x,t) + \sqrt t > -(N  + \frac {\al}{2})$, which is a contradiction. 
Hence, there is no such $(x,t)$.

\end{proof}

\section{A local result in three dimensions}\label{localsec}
In this section we restrict ourselves to the three dimensional case.
We first consider  a ball of radius $1$ in a three dimensional manifold
which has curvature operator and volume  bounded from below by known
constants at time zero.
For later times we assume a bound on the curvature of the form
$| \Riem(g(t)) | \leq \frac N t$ on the  time $t$ ball of radius $1$,
where $N$ depends on the curvature bound from below.
We show that the curvature can not become too negative too quickly on
smaller balls.

\begin{theorem}\label{thred}
Let $ v_0>0$ and $N\geq 1$ be given.
Let $(M^3,(g(t)) _{t \in [0,T)}$ be a smooth complete solution to Ricci flow with $T\leq 1$, and let $x_0 \in M$  be a point
such that
\begin{itemize}
\item $\vol{{}^{0} B_{1}(x_0)} \geq v_0$ and
\item $\curlR(g(0)) \geq - \frac {1}{400 N} $ on  ${{}^{0} B}_{1}(x_0)$,
\item $ |\Riem(g(t))|(x) \leq \frac N { t} $ for all $t \in (0,T)$ for
  all $x \in  {{}^{t} B}_{1}(x_0)$
\end{itemize}
Then, for all $1>\si >0$  there exists $\de = \de(v_0,N,\si) >0$ and $\ti
v_0 = \ti v_0(N,v_0,\si) >0$ such that
\begin{itemize}
\item $\vol({ {}^{t} B}_{1}(x_0)) \geq  \ti v_0 $
\item $\curlR(g(t)) \geq -  1 $ on  $ {}^{t} B_{(1-\si)}(x_0).$
\end{itemize}
as long as $t \leq (\de_0)^2$  and $t \in [0,T)$.
\end{theorem}



\begin{proof}  
By the Bishop-Gromov volume comparison principle we have 
$\vol(   {{}^0 B}_{s}(x_0)) \geq c(N,v_0)s^n $ for all $s \leq 1$.
In particular 
$\vol( {{}^0 B}_{\frac 1 {1000} } (x_0)) \geq \ti v_0(N,v_0)>0 $.

For some maximal time interval $[0,S_{max})$, $S_{max} \leq T$ we have (due to smoothness) that 
 \begin{itemize}
 \item[(a)]  $\curlR \geq - 1 $ on ${{}^tB}_{(1-\si)}(x_0)$
\item[(b)] $\vol({{}^tB}_{\frac 1 {1000}}(x_0)) \geq \frac{\ti v_0}{2} $.
\end{itemize} 

Our aim is to obtain an estimate from below for the time $S$ which
only depends on $v_0$ and $N$ ($n = 3$ is fixed here).

For convenience we denote the constant  $\frac {1}{10N} $ by
$\ep_0:=\frac {1}{10 N} $.

{\bf Claim (i)  } The scalar curvature is bounded from below by $-2\ep_0$ on
${{}^tB}_{(1-\frac {\si}{10})}(x_0)$
for $t \leq S$ where $S = S(v_0,\si, N) >0$, as long as $t \leq
S_{max}$.

\noindent proof of Claim (i):
\noindent
This is proved using the same argument as that given in the proof of   Claim (i) in the proof of  Theorem
\ref{2dd} above.

This finishes the proof of the {\bf Claim (i)}.

For convenience we introduce $\al < \be < \ga$ to be the eigenvalues of $\curlR$ as in Hamilton.
Then $\Sc = \al + \be + \ga$ and
$|\Ricci|^2 = \frac 1 2( \al^2 + \be^2 + \ga^2 + \al \be + \al \ga + \be
\ga).$

{\bf Claim(ii)}
$\al +  2\Sc \geq - 2\ep_0 $ on
${{}^tB}_{(1-\frac{\si}{5})}(x_0)$ as long as $t \leq S_{max}$ and $t \leq  S(N,\si,v_0)$.

\noindent {\it Remark:} A local result of this type was first shown by
B.L-Chen under the extra assumption that the Ricci curvature remains
bounded on ${{}^tB}_1(x_0)$: see Proposition 2.2 in \cite{BLC}  ( for a
related result see Lemma 4.1 of \cite{Si1}). 

\noindent proof of Claim (ii):

\noindent Let $ L(V,V) = \Riem(V,V) + 2\Sc \Id(V,V)$, 
$P(V,V) = h  L(V,V)  + \ep_0(1 + \frac{k t }{\ep_0}  ) \ \ (  = h(\Riem(V,V) + 2\Sc
\Id(V,V))  + \ep )$ for 2-Forms, where we have
chosen now $\phi$ to satisfy $\phi(r) = 1$ for all $r \leq (1-\frac{\si}{6}) $ and
$\phi(r) = 0$ for all $r\geq (1-\frac{\si}{7}) $ in the definition of $h$ in
Proposition \ref{cutty}.
We shall only be concerned with points where $h \neq 0$, and so we
may use freely the results of Claim (i) for $t\leq S(N,v_0,\si)$. We do
so, sometimes without further comment. $k= k(N,\si,v_0)$ is a large constant which we shall choose later
in the proof.

Also we have introduced the notation $\ep =\ep(t) =  \ep_0(1 + \frac{k t
}{\ep_0}  )$.
For the time intervals we are considering, we have 
$\ep_0\leq \ep(t) \leq 2 \ep_0$, as we shall assume that $t \leq
\frac{\ep_0}{2k}$.

In all of the following arguments (also for the proofs of claims
(iii),(iv) and (v)) we shall be calculating the evolution of the
curvature in the setting of \cite{HaFour}. That is, we are using the trick of
K.Uhlenbeck.
In particular, the metric $G_{ab}(x) := g_{ij}(x,t)u^i_a(x,t) u^j_b(x,t)$
is the pullback of the metric $g(x,t)$, and it is time independent:
$\partt G_{ab}(x) = 0$.
$\Id$ here is the operator on two forms given by $\Id(V,W):=
G^{ab}G^{cd}V_{ac}W_{bd}.$
In particular $\partt (\Id(V,W)) = 0$ for a time independent vector field.
The connection,${{}^{t}\grad}$,  is the pullback connection of $
{{}^{g(t)}\grad} $. We still have ${{}^t\lap} f(x) = \lap_{g(t)}
f(x)$ for smooth functions $f:M\to \R$ (the left hand side is the laplacian with respect to the
pullback connection and the right hand side is the laplacian with respect to
$g(t)$).
We also have ${{}^t\grad} \Id = 0$. See \cite{HaFour}  for details.
Once again we consider only $t \leq
\frac{A^2}{100m_0^2} =\ti m_0(\si,N,v_0)$ so that Proposition \ref{cutty} is applicable.
Then $P(V,V) =h( \al + 2\Sc) + \ep$ for a 2-form $V$ with length one
which minimises $P$ at any point in space and time.
We first estimate the reaction term coming from the evolution equation
for $L$.
At the end of the proof we explain how to deal with the reaction
diffusion equation for $P$ (in particular the gradient terms).
In Lemma 4.1 of \cite{Si1} it is shown  (with $\ep  := 1$ there )
that the reaction equation for $L = \al + 2R$ is given by
\begin {eqnarray*}
\partt (\al + 2\Sc)  
&& = \al^2 + \be \ga + 2( \al^2 + \be^2 + \ga^2 + \al \be + \al \ga  +
\be \ga )\cr 
\end{eqnarray*}

In case $\be, \ga \geq 0$ ,  or $\be,\ga \leq 0$  (which implies $\be
\ga \geq 0$) we get 
\begin {eqnarray*}
\partt (\al + 2\Sc)  && \geq  \al^2 +  2( \al^2 + \be^2 + \ga^2 + \al \be + \al \ga 
)\cr 
&&= 3 \al^2 +  2\be^2 + 2\ga^2 + 2 \al \be + 2\al \ga  \cr
&& \geq ( \al^2 + \ga^2 + \be^2)\cr
&&\geq \frac 1 {1000}(\al + 2\Sc)^2
\end{eqnarray*}
in view of Young's inequality.
In case $\be \leq 0, \ga \geq 0$ ( which implies that $\al \be \geq
0$)
we get by applying Young's inequality a number of times
\begin {eqnarray*}
\partt (\al + 2\Sc)    && = \al^2 + \be \ga + 2( \al^2 + \be^2 + \ga^2 + \al \be + \al \ga  +
\be \ga )\cr 
&&\geq 
 3\al^2 + 2\be \ga +  2\al \ga +   2\be^2 + 2\ga^2 
)\cr 
&& \geq 2\al^2 + 2\be\ga + 2\be^2 + \ga^2 \cr
&&\geq \frac 1 3 ( \al^2 + \be^2 + \ga^2) \cr
&& \geq \frac 1 {1000} (\al + 2\Sc)^2 \cr
\end{eqnarray*}

At a first time and point $(y,s) $ where $h (\al + 2\Sc) = -\ep$, we must clearly
have that $\al <0$ ( otherwise $ -\ep = h(\al + 2\Sc) \geq 0 $ which is a
contradiction ).
Let $V$ be a 2-form with length one such that $P(V,V) = 0$.
We have 
\begin{eqnarray*}
\partt P(V,V) && \geq (\lap P)(y,s)(V,V) - 2(g^{ij}\grad_j h)(y,s) (\grad_i L)(y,s)(V,V)  \cr
&& \ \    + \frac 1 {1000}h(y,s)(L(y,s)(V,V))^2 + k 
\end{eqnarray*}
in view of the above reaction equation for $L$.

We estimate the gradient term in the above as follows
\begin{eqnarray}
 - 2g^{ij}\grad_j h(y,s) (\grad_i L)(y,s)(V,V)  && =
 - \frac{2}{ h(y,s)} (g^{ij} \grad_j h)(y,s)  (\grad_i (hL))(y,s)(V,V)
 \cr&& 
\ \  \ + 2L(y,s)(V,V) \frac{|\grad h|^2}{h}(y,s) \cr
&& =   2L(y,s)(V,V) (\frac{|\grad h|^2}{h})(y,s) \cr
&& \geq -\frac{L^2}{2000}(y,s) h(y,s) - \frac {4000}{h^3(y,s)}    |\grad
h|^4(y,s) \cr
&& \geq - \frac{L^2}{2000}(y,s) h(y,s) - 2\hat C  \label{doodly2}
\end{eqnarray}
where in the last line we have once again used that
$|\grad h|^4 \leq \hat C(\si) h^3$.
Hence we obtain 
\begin{eqnarray}
\partt P(V,V) && \geq (\lap P)(y,s)(V,V) - 2\hat C+ k \label{doodly}
\end{eqnarray}
at $(y,s)$, which leads to a contradiction if $ k$ is chosen
appropriately (here $n=3$).
Hence $P$ remains non-negative in the time interval considered.
In particular, using the definition of $h$ , we see that
$h(\al + 2R) +\ep_0(1 + \frac{kt}{\ep_0}) \geq 0$ implies
$$ \al + 2R + 2\ep_0 \geq 0 $$ on ${{}^tB}_{(1- \frac{\si}{5}) } (x_0)$
 for $t \leq S(N,v_0,\si)$ ( possibly a smaller $S$ now ) as required.

\noindent This finishes the proof of Claim (ii).

\noindent {\bf Claim (iii)} The volume condition (b) is not violated for a well
defined time interval $ t \leq S(N,v_0,\si)$, as long as $t \leq S_{max}$.

The proof may be taken from Claim (ii) in the  proof of Theorem
\ref{2dd} above, with two changes:  we use $\curlR\geq -1$ on
in place of $\Sc \geq -1$ and we use
the assumption that $ |\Riem(g(t))| \leq \frac N  t$ on ${{}^tB}_{1} (x_0)$

This finishes the proof of  {\bf Claim (iii)}.

{\bf Claim (iv)} The curvature condition (a) will also not be violated
for a well defined time interval $[0,S(N,v_0,\si))$ as long as $t\leq S_{max}$.
The proof of this claim is initially similar to that of Claim (i) and Claim (ii).
In order to estimate the gradient term we require some different arguments.

Define $\ep(t):=  \ep_0(\frac 1 2 +\frac{kt^{\frac 1 4}}{\ep_0})$.
Let $Y := h \curlR +  (\frac{\ep}{ 100}  +
\ep t \Sc )\Id$  ( note here: $\ep t = \ep(t) t$ ), where $h$ is a cut-off function
coming from Proposition \ref{cutty} with $\phi(r) = 1$ for all $r \leq 1- \frac
\si 2$ and
$\phi(r) = 0$ for all $r \geq  1- \frac
\si 3$. This is a local version of the tensor appearing in Lemma 5.2 of
\cite{Si1}.
$k = k(N,\si,v_0)$ is a large positive constant which shall be chosen later in the proof. 
We wish to show that $Y$ remains larger than zero for $t \leq
S(N,\si,v_0)$ in $  {{}^tB}_{(1-\frac {\si}{4})}(x_0)$
as longs as $t\leq S_{max}$. Assuming we have a first
time and point $(y,s)$, $ y \in   {{}^tB}_{(1-\frac {\si}{5})}(x_0)$
and a two form $V$ of length one where
$Y(y,s)(V,V) = 0$, then we must have $h(y,s) >0$, otherwise
$Y= h \al + (\ep t \Sc + \frac{\ep}{100})  = (\ep t \Sc + \frac{\ep}{100})
\geq -4\ep_0 \ep s + \frac{\ep}{100} >0$ for $s \leq S(N,v_0,\si)$ small in view of Claim (i), which is a contradiction.
Henceforth, we shall only be concerned with points where $h > 0$ and so we may
freely use the results of both Claims (i) and (ii) for $t\leq
S(N,v_0,\si)$  in
view of the definition of $\phi$ we have chosen here. We do so, sometimes
without further comment.
We assume that $t^{\frac 1 4} \leq \frac{\ep_0 }{100k}$ so that
$\frac 1 2 \ep_0 \leq \ep \leq 2 \ep_0$.
First we examine the reaction term occurring in the evolution of the
tensor $h \curlR +  (\frac{\ep}{ 100}  +
\ep t \Sc )\Id$. 
Afterwards we explain how to deal with the reaction diffusion equation occurring here, in particular
how to estimate the gradient terms and the zero order term which
appears at the end of this estimate.
For convenience we introduce $\al < \be < \ga$ to be the eigenvalues of $\curlR$ as in \cite{HaFo}.
Then $\Sc = \al + \be + \ga$ and
$|\Ricci|^2 = \frac 1 2( \al^2 + \be^2 + \ga^2 + \al \be + \al \ga + \be \ga).$
It is shown in \cite{HaFo} that the reaction equation for $\al$ is
given by
$\partt \al = \al^2 + \be\ga$.
We have an evolution inequality for $h$ given by $\partt h \leq  \lap h$ for
$t \leq k$ and hence the reaction equation for $h$ may be estimated by 
$\partt h \leq  0$ for $t \leq S(N,v_0,\si)$.  This means that the reaction equation for $h
\al$ at a point in space and time where $\al <0$ and  $\ga >0$  may be
estimated by 

\begin{eqnarray*}
\partt \al h && \geq  h\al^2 + h\be\ga \cr
&& \geq h\al^2 + h\al\ga \cr
&& = h\al^2 + (h\al + \frac{\ep}{ 100} +\ep t \Sc )\ga
- (\frac{\ep}{ 100} +\ep t \Sc )\ga.
\end{eqnarray*}
 If $\ga \leq 0$ then $0 \geq \ga, \be, \al \geq -2\ep_0$ in view of Claim
 (i) and hence we have 

\begin{eqnarray*}
\partt \al h && \geq  h\al^2 + h\be\ga \cr
&& \geq h\al^2\cr 
&& \geq  h\al^2 + h \al \ga -10 \cr
&& =  h\al^2 + (h\al + \frac{\ep}{ 100} +\ep t \Sc )\ga
- (\frac{\ep}{ 100} +\ep t \Sc )\ga - 10.
\end{eqnarray*}

The reaction
equation for  $(\frac{\ep}{ 100}  +
\ep t \Sc )\Id$ is  
\begin{eqnarray*}
\partt (\frac{\ep}{ 100}  + \ep t \Sc )  &  = &  \frac{k}{400 t^{\frac 3 4}}  + \ep \Sc+
2\ep t |\Ricci|^2 + \frac{kt R}{4t^{\frac 3 4}} 
 \cr
&  \geq &  
\frac{k}{500 t^{\frac 3 4}}  + \ep \Sc + \ep t ( \al^2 + \be^2 + \ga^2 + \al \be + \al \ga +
\be \ga),
\end{eqnarray*}
for $t \leq S(N,v_0,\si)$, in view of Claim (i). Combining these three inequalities we get
\begin{eqnarray}
\partt ( \al h + ( \frac{\ep}{ 100}  + \ep t \Sc )) \geq &&
 h\al^2 + (h\al + \frac{\ep}{ 100} +\ep t \Sc )\ga 
+ \ep [ \Sc -  \ga \frac{1}{ 100} ] \cr
&&  + \ep t [ - \Sc \ga +  \al^2 + \be^2 + \ga^2 + \al \be + \al \ga +
\be \ga] + \frac{k}{800 t^{\frac 3 4}} \cr
= &&  h\al^2 + (h\al + \frac{\ep}{ 100} +\ep t \Sc )\ga \cr
&& + \ep [ \Sc -  \ga \frac{1}{ 100} ]
 + \ep t [   \al^2 + \be^2  + \al \be ] + \frac{k}{800 t^{\frac 3 4}}
 \cr
\geq  && h\al^2 + (h\al + \frac{\ep}{ 100} +\ep t \Sc )\ga 
+ \ep [ \Sc -  \ga \frac{1}{ 100} ] + \frac{k}{800 t^{\frac 3 4}} \label{bububu} 
\end{eqnarray}

Assuming we have a first time and point $(y,s)$ where 
$Y(y,s) = 0$, $ y \in  {{}^tB}_{(1-\frac {\si}{5})}(x_0)$ then we
must have $h(y,s) >0$ as we explained at the
beginning of the proof of this Claim.
At $(y,s)$ we have $\al(y,s)  = - \frac 1 h (  s\ep(s) \Sc(y,s) +
\frac{\ep(s)}{ 100}) <  \frac 1 h (4s\ep(s) \ep_0)  - \frac 1 h \frac{\ep(s)}{100})  <0$ ( in view of Claim (i)) and hence the reaction equation of 
 $Y = h \al + \ep t \Sc + \frac{\ep}{100}$ may be estimated by
\begin{eqnarray*}
\partt ( \al h + ( \frac{\ep}{ 100}  + \ep t \Sc )) && \geq  h\al^2 + (h\al + \frac{\ep}{ 100} +\ep t \Sc )\ga
+ \ep [ \Sc -  \ga \frac{1}{ 100} ] + \frac{k}{800 t^{\frac 3 4}} \cr
&&\geq  h\al^2 
+ \ep [ \Sc -  \ga \frac{1}{ 100} ] + \frac{k}{800 t^{\frac 3 4}} ,
\end{eqnarray*}

in view of the estimate \eqref{bububu},
where we have used that
$Y(y,s) = 0$.
 $ [ - \frac 1 {100}\ga   +  \Sc] =      [\frac {(\al + \be)}{100} + \frac{99}{100}
\Sc ] \geq [\frac 2 {100} \al + \frac 4 {100} \Sc + \frac {95}{100} \Sc ] \geq \frac 2 {100}
[\al + 2\Sc] -\frac{95}{100}\ep_0 \geq -10\ep_0$ 
in view of Claim (i) and (ii).
Hence  
\begin{eqnarray*}
 \partt Y \geq &&  h(y,s) \frac{ \al^2 }{2}+ \frac{k}{801 t^{\frac 3 4}}  \cr
\end{eqnarray*}

Now we examine the reaction diffusion equation. Using the estimate on the reaction equation above, we see that
at time $(y,s)$ in direction $V$ where $Y(y,s)(V,V) = 0$ 
\begin{eqnarray}
\partt Y(V,V) && \geq (\lap Y)(y,s)(V,V) - 2(g^{ij} \grad_j h)(y,s)
(\grad_i  \curlR) (y,s)(V,V)  \cr
&& \ \    + h(y,s)\frac 1 2\al^2(y,s) + \frac{k}{801 t^{\frac 3 4}} \label{doodly3}
\end{eqnarray}

We estimate the second term  (the gradient term) of the right hand side of this inequality:
\begin{eqnarray}\label{bubu1}
&&  - 2(g^{ij} \grad_j h)(y,s)
(\grad_i  \curlR) (y,s)(V,V) \cr
&& =  - \frac 2 {h(y,s)} g^{ij}\grad_j h  \grad_i (\curlR  h) (y,s)(V,V)
+ 2\al(y,s)\frac{|\grad h|^2}{h} \cr
&& =   - \frac 2 {h(y,s)} g^{ij}  \grad_j h (y,s)  \grad_i [\curlR  h +   \frac{\ep}{ 100} \Id +
\ep t \Sc \Id ](y,s)(V,V)\cr
&& \ \ \ + 2\al(y,s) \frac{|\grad h|^2}{h} 
+  \frac {2 \ep s} {h(y,s)} g^{ij}(\grad_i \Sc \grad_j h)(y,s)(V,V)
\cr
&& =  2\al(y,s) \frac{|\grad h|^2}{h} 
+  \frac {2 \ep s} {h(y,s)} g^{ij}(\grad_i \Sc \grad_j h)(y,s)(V,V) \cr
&& \geq  - \frac 1 4 \al^2(y,s)h(y,s)  -   2\frac{|\grad h|^4}{h^3} (y,s)
-  \frac{2\ep s}{h(y,s)} |\grad \Riem |(y,s) |\grad h|(y,s) \cr
&& \geq  -  \frac 1 4 \al^2(y,s)h(y,s)  - C(\si) -   \frac{2\ep s}{h(y,s)} |\grad \Riem |(y,s) |\grad h|(y,s),
\end{eqnarray}

since $|\grad \phi|^2 \leq \phi^{\frac 3 2}C(\si)$ .
Using the estimates of Shi, see \cite{HaFo} Theorem 3.1, and the fact that
$|\Riem |(x,t) \leq \frac N t$,  we see that
$|\grad \Riem|^2 \leq \frac{ \hat c N^3}{t^3}$  where $\hat c = \hat
c(N,\si)$ . We explain this in more detail. $x \in {{}^tB}_{(1- \frac {\si} {5} ) }(x_0)$ implies ${{}^tB}_{ \frac {\si}
  {1000}}(x)   \subseteq  {{}^tB}_1(x_0)$ for the $x$ we are
considering.  We work in the ball ${{}^tB}_{ \frac {\si}
  {1000}}(x) $, and we have $|\Riem(\cdot,t)| \leq \frac N t$ there.
Scale  so that $t =1$.
 Note that $|\Riem|\leq 2N$ for $t \in
(\frac 1 2,1]$ after scaling, so distances change in a controlled
manner near time $1$. 
This allows one to find a parabolic region of
the form ${}^{g(1)}B_r(x) \times [-r^2,1]$ for some $r = r(\si,N) >0$
close to one on which $|\Riem|\leq 2 N $.
Now we may use the estimates of Shi, see \cite{HaFo} Theorem 3.1,  at $t=1$, and
then scale back to $t$ ( this completes the explanation of the
estimate $|\grad \Riem|^2 \leq \frac{ \hat c N^3}{t^3}$  ).
Using $|\grad \Riem|^2 \leq \frac{ \hat c N^3}{t^3}$, we get
\begin{eqnarray*}
\frac{2\ep s }{h(y,s)} |\grad \Riem |(y,s) |\grad h|(y,s)
\leq \frac{4 \hat c N^{\frac 3 2}\ep_0 }{h(y,s) s^{ \frac 1 2} } |\grad h|(y,s).
\end{eqnarray*}
Notice that we must have $ h(y,s)\geq \frac{ \ep_0 s}{N 500 }$. If not,
then
\begin{eqnarray*}
 h \al + (\ep s \Sc + \frac{\ep}{100})\Id && \geq -h|\Riem |  - \frac{\ep_0}{400} +
 \frac{\ep_0}{200} \cr
 && \geq -h|\Riem |  + \frac{\ep_0}{400}
 \cr
&& \geq -\frac{N \ep_0}{N  500  } +\frac{\ep_0}{400} \cr
&& >0,
\end{eqnarray*}
( at $(y,s)$ ) if $s \leq S(N,v_0,\si)$, which is a contradiction ( here we have used
again that $ |\Riem(g(t))| \leq \frac{N}{t}$
and $\ep(s)s\Sc \geq - \frac{\ep_0}{400}$ for $s\leq S(N,\si,v_0)$ small enough in
view of Claim (i) ).

Hence 
\begin{eqnarray*}
 \frac{2\ep s}{h(y,s)} |\grad \Riem |(y,s) |\grad h|(y,s)
&& \leq  \frac{4 \hat c N^{\frac 3 2}\ep_0 }{h(y,s) s^{ \frac 1 2} }
|\grad h|(y,s) \cr
&& \leq  \frac{4 c \hat c N^{\frac 3 2}\ep_0 h^{ \frac 3 4} }{h(y,s) s^{ \frac 1 2} } \cr
&& =  \frac{4 c \hat c N^{\frac 3 2}\ep_0 }{h^{\frac 1 4}(y,s) s^{ \frac 1 2} } \cr
&& \leq \frac{ 4 (500 N )^{\frac 1 4}  c \hat c N^{\frac 3 2} \ep_0 }
{s^{\frac 3 4} \ep^{\frac 1 4}_0 } \cr
&& \leq \frac{ \hat c \ep^{\frac 3 4}_0 c(500 N)^{2}}{s^{\frac 3 4}},
\end{eqnarray*}
where we have once again used that
$|\grad h|^4 \leq C h^3$.

Substituting this inequality into \eqref{bubu1}, we obtain
\begin{eqnarray}\label{bubu2}
-2g^{ij} \grad_j h  \grad_i \curlR(y,s)(V,V) 
&& \geq  -  \frac 1 4 \al^2(y,s)h(y,s)   - C -  \frac{ \hat c
  \ep^{\frac 3 4}_0 c(500 N)^{2}}{s^{\frac 3 4}}.
\end{eqnarray}

Substituting the inequality \eqref{bubu2} into \eqref{doodly3} we see
that
\begin{eqnarray*}
\partt Y(y,s)(V,V)  > (\lap Y)(y,s)(V,V)  
\end{eqnarray*}
if $k= k(N,\si,v_0)$ is chosen large enough. This contradicts
$(y,s)$ being a first time and point where $Y$
is zero.
This implies that  $Y$ remains larger than zero for a well defined
time interval $[0,S(N,v_0,\si))$ as long as $t \leq S_{max}$.

Using the definition of $\phi$ (which is used in $h$) we see that
$$\al \geq - \frac {2\ep}{100} - 2\ep t \Sc \geq -4\ep_0 N >-1$$ on ${{}^tB}_{(1-\si)}(x_0)$
for $ t \leq S(N,v_0,\si)$ as long as $t \leq S_{max}$.
Here we have used $|\Sc(\cdot,t)|t \leq N$.   

That is, condition (a) will also not be violated for a well defined
time interval. That is $S_{max} \geq S(n,N,v_0)>0$.
This finishes the proof of {\bf Claim (iv)} and the proof of the theorem if one accepts Claim (v)
(note that $(\ti b)$ also holds).

\noindent{\bf Claim (v)} If $d_t(\cdot)$ is not
differentiable at $x \in M$ then we use the trick of E.Calabi
(\cite{Ca}) as explained in the proof of Claim (iii) of  the proof of
Theorem \ref{2dd}.

Hence we may argue with
$ \ti h(y,s):= e^{-2k_0s} \ti k(y,s)$ 
everywhere above.
If for example $(x,t)$  is a first time and point where
$(h (\curlR + 2\Sc \Id)  + \ep \Id)(V,V) = 0$ ,  then $(x,t)$  is a first time and point where
$(\ti h (\curlR + 2\Sc \Id)  + \ep \Id)(V,V) = 0$ , at least locally ( see the proof of Claim (iii) of the proof of
Theorem \ref{2dd} for further details ).
Arguing as in the proof of Claim (iii) of the proof of
Theorem \ref{2dd} , we see that such an $(x,t)$ cannot exist.

A similar argument holds for the tensor in Claim (iv) of this proof.

This finishes the proof of Claim (v) of this proof and the proof of the theorem.
\end{proof}

\begin{remark}\label{endremark}
Theorem \ref{thred} establishes  Theorem \ref{threed} for the case $V
= \frac{1}{400N}$, 
$N \geq 1$. The case of general $V>0$ may be obtained as follows.
Scale so that $r=1$.
Now scale again, so that $\curlR \geq -\frac{1}{400N}$ on $ {{}^0
  B}_r(x_0)$ where $r = \sqrt{V  400 N} $ ( notice that we require $V>0$ to
do this ).
We now have $\vol ({ {}^0B_r}(x_0)) \geq \hat v_0 = \hat v_0(V,N,v_0)
>0$.
Now repeat the proof of Theorem \ref{thred} with the following changes: replace $v_0$ by $\hat v_0 $, replace all
balls $^t B_s(x_0)$ that appear in the proof by $^t B_{rs}(x_0)$,
choose a cut-off function $\phi$ from Proposition \ref{cutty} with
$\phi(s) = 1$ for all $s \leq r(1- \frac{\si}{6})$
and $\phi(s) = 0$ for all $s \geq r(1- \frac{\si}{7})$ in Claim (ii)
respectively 
with $\phi(s) = 1$ for all $s \leq r(1- \frac{\si}{2})$
and $\phi(s) = 0$ for all $s \geq r(1- \frac{\si}{3})$ in Claim (iv).
The proof then works without any further changes, except that the
constants that occur now also depend on $V$ ( this dependence also
appears in the
statement of the Theorem ).
\end{remark}

\bibliographystyle{amsplain}
 \bibliography{hyp}

\begin{thebibliography}{10}



\bibitem[1]{Ca}
Calabi, E. ,\emph{ An extension of E.Hopf's maximum principle with an
application to Riem. Geometry}, 
Duke Math J., (1958), 45-56






\bibitem[2]{CaChChYa}
Cao, H.D., Chow, B., Chu,S.C., Yau,S.T. ( Editors )
\emph{ Collected papers on the Ricci flow},
 Series in Geometry and Topology, Vol.37, International Press

\bibitem[3]{ChTaYu}
Chau, A., Tam, L-F., Yu, C.
\emph{Pseudolocality for the Ricci flow and applications}
Can.J.Math 63(2011), 55-85

\bibitem[4]{BLC}
Chen, B.-L., \emph{Strong uniqueness of Ricci flow},
J Dfferential Geometry, 82 (2009) 363-382

\bibitem[5]{CCGGIIKLLN}
Chow,B., Chu, S.-C., Glickenstein, D., Guenther, C., Isenberg,J.,
Ivey,T.,Knopf,D.,Lu,.P,Luo,F.,Ni,L.,
\emph{The Ricci flow: Techniques and applications, Part III:
  Geometric-Analytic aspects}
AMS , Math. Surveys and Monographs, Vol.163, (2010).

\bibitem[6]{ChGuZh}
Chen, B.-L., Xu, G., Zhang,Z.\emph{Local pinching estimates in 3d Ricci flow}
arXiv:1206.1814 (2012)

\bibitem[7]{HaFo}
Hamilton,R.S.
\emph{The formation of singularities in the Ricci flow},
{Collection: Surveys in differential geometry}, 
Vol. II (Cambridge, MA), 7--136, (1995).




\bibitem[8]{HaFour}
Hamilton,R.S.
\emph{Four manifolds with positive curvature operator}
J.Diff.Geom. 24, no.2 , 153- 179 (1986)

\bibitem[9]{HaCo}
Hamilton,R.S.
\emph{A compactness property of the Ricci Flow}
American Journal of Mathematics, 117, 545--572, (1995)

\bibitem[10]{KL}
Kleiner,B., Lott,J.
\emph{Notes on Perelman's papers}
Geometry and Topology 12 (2008), 2587-2855
(Arxiv version, arxiv: math/0605667)

\bibitem[11]{Lu}
Lu, P.
\emph{A local curvature bound in Ricci flow}
Geometry and Topology 14 (2010), 1095-1110




\bibitem[12]{Pe1}
Perelman,G.,
\emph{The entropy formula for the Ricci flow and its geometric applications},
MarthArxiv link: math.DG/0211159, 2002

\bibitem[13]{Pe2}
Perelman,G.,\emph{Ricci flow with surgery on three manifolds}, MarthArxiv link: math.math.DG/0303109,2003

\bibitem[14]{SchSi} Schulze, F., Simon,M. 
\emph{Expanding solitons with non-negative curvature operator coming out of cones}, 
arXiv:1008.1408, (2010)


\bibitem[15]{Si1} Simon, M. \emph{Ricci flow of almost non-negatively curved three manifolds},
 Journal f\"ur reine und angew. Math. \textbf{630}  (2009), 177-217


\bibitem[16]{Si3} Simon,M. \emph{Ricci flow of non-collapsed three manifolds
whose Ricci curvature is bounded from below},
 J. reine angew. Math. 662 (2012), 59—94


\bibitem[17]{Wa}
Wang, Y. 
\emph{Pseudolocality of Ricci Flow under Integral Bound of Curvature
  }, arXiv:0903.2913






\end{thebibliography}
\def\cprime{$'$}
\providecommand{\bysame}{\leavevmode\hbox to3em{\hrulefill}\thinspace}
\providecommand{\MR}{\relax\ifhmode\unskip\space\fi MR }
\providecommand{\MRhref}[2]{%
  \href{http://www.ams.org/mathscinet-getitem?mr=#1}{#2}
}
\providecommand{\href}[2]{#2}

\end{document}